\documentclass{article}

\usepackage[utf8]{inputenc}
\usepackage[T1]{fontenc}

\usepackage{amsmath}
\usepackage{amssymb}
\usepackage{amsthm}

\usepackage{microtype}
\usepackage{datetime}
\usepackage{fancyhdr}
\usepackage{tikz}
\usetikzlibrary{positioning,calc,fadings}

\usepackage{multirow}
\usepackage{booktabs}

\newcommand{\noop}[1]{}

\usepackage{graphicx}

\theoremstyle{plain}
\newtheorem{theorem}{Theorem}[section]
\newtheorem{proposition}[theorem]{Proposition}
\newtheorem{lemma}[theorem]{Lemma}
\newtheorem{corollary}[theorem]{Corollary}

\theoremstyle{definition}
\newtheorem{definition}[theorem]{Definition}
\newtheorem*{example}{Example}
\newtheorem{remark}[theorem]{Remark}

\newcommand{\zmod}[1]{\mathbb{Z}/_{\!\!\;\!#1}}
\newcommand{\TC}[2]{TC^{#1,#2}}
\renewcommand{\P}{\mathcal{P}}
\newcommand{\cat}{\operatorname{cat}}

\newcommand{\nil}{\operatorname{nil}}
\newcommand{\nilker}{\nil\ker}
\newcommand{\secat}{\operatorname{secat}}
\newcommand{\im}{\operatorname{im}}
\newcommand{\id}{\operatorname{id}}
\newcommand{\const}{\operatorname{const}}

\newcommand{\suchthat}{\;\big|\;}

\title{Effective topological complexity of spaces with symmetries}
\author{Zbigniew Błaszczyk, Marek Kaluba\footnote{
The authors have been supported by the National Science Centre, respectively under grants 2014/12/S/ST1/00368 and 2015/19/B/ST1/01458.}}
\date{}

\begin{document}
\maketitle

\begin{abstract}
We introduce a version of Farber's topological complexity suitable for investigating mechanical systems whose configuration spaces exhibit symmetries. Our invariant has vastly different properties to the previous approaches of Colman--Grant, Dranishnikov and Lubawski--Marzantowicz. In particular, it is bounded from above by Farber's topological complexity.
\end{abstract}

\section{Introduction}

A \textit{motion planner} in a space $X$ is an algorithm which, given a pair of points $(x,y) \in X \times X$, outputs a path from $x$ to $y$ in $X$. This notion is usually considered in the context of robotics, where $X$ is taken to be the space of all states (``configuration space'') of a mechanical system. Topology on $X$ is such that paths produced by the motion planner are continuous. A path can then be interpreted as a movement of the robot from one state to the other.

One would hope for a motion planner that is stable in the sense that a minor change of either the initial or terminal state results in a predictable change of the path taken by the robot. This, however, is rarely possible; see \cite[Theorem~1]{Farber2003a}. In order to quantify the ``order of instability'' of configuration spaces of mechanical systems, Farber \cite{Farber2003a}, \cite{Farber2004} introduced the notion of topological complexity ($TC$). Due to its applications in topological robotics and close relation to Lusternik--Schnirelmann category, it has attracted plenty of attention.

Mechanical systems often come equipped with symmetries visible in their configuration spaces, thus it not surprising that there have been attempts at weaving symmetries into the definition of to\-po\-lo\-gi\-cal complexity before, vide Colman--Grant \cite{Colman2012}, \mbox{Dranishnikov~\cite{Dranishnikov2015}}, and Lubawski--Marzantowicz \cite{Lubawski2014}.
They are, however, useful if one is either interested in motion planning algorithms that are somehow symmetric, or simply seeks for new invariants that are interesting from mainly mathematical, as opposed to robotical, point of view. Our foundational idea is completely different. We believe that symmetries present in configuration spaces can be used to ease the task of motion planning. Consider the following example.

\begin{example}
Assume that the task of the mechanical arm $R$ below is to grab and screw in an object, and in order to do so, it rotates its ``head'' around an axis contained in the plane of the figure.
\begin{center}
\includegraphics[scale=0.6]{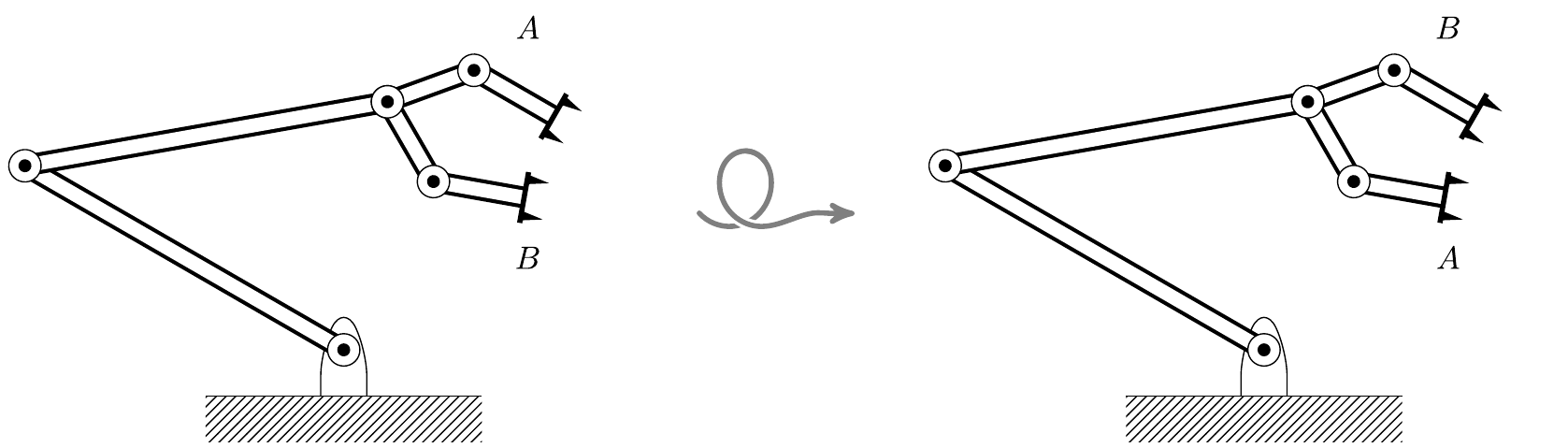}
\end{center}
It is easy to imagine that the arm can be designed in a way which makes the exact position of pliers irrelevant: the object in question can be grabbed equally well from the positions $(A,B)$, $(B,A)$, and all in between. This $S^1$-\hspace{0pt}symmetry can be interpreted as follows: even though these positions are physically different states of $R$, they are functionally equivalent, and therefore teaching the robot how to move from one to another is a waste of effort. While planning the motion, an algorithm should be allowed to ``leap'' between the states.
\end{example}

With this sort of example in mind, we introduce a new $G$-\hspace{0pt}homotopy invariant, \textit{effective topological complexity} $\TC{G}{\infty}$, designed to measure ``order of instability'' of configuration spaces of mechanical systems with symmetries. It is a modification of the Lubawski--Marzantowicz approach and its standout property is that it is bounded from above by~$TC$. In fact, a more general phenomenon occurs: if $H\subseteq G$ is any subgroup, then $\TC{G}{\infty}\leq  \TC{H}{\infty}$ (Lemma \ref{lemma:monoticity}). The new invariant also enjoys a cohomological lower bound and a product inequality resembling those of $TC$ (Theorems \ref{thm:effectiveTC_lower_bound} and \ref{thm:product_inequality}, respectively).

\section{Preliminaries}\label{section:preliminaries}
\subsection{Topological complexity}\label{section:TC}

Write $PX$ for the space of continuous paths in a topological space $X$. The map $\pi_1 \colon PX \to X \times X$ given by
\[ \pi_1(\gamma) = \big(\gamma(0), \gamma(1) \big) \textnormal{ for $\gamma \in PX$}\]
is well-known to be a fibration. A \textit{motion planner} on an open subset $U \subseteq X \times X$ is a section of $\pi_1$ over $U$, i.e. a map $s \colon U \to PX$ such that $\pi_1 \circ s$ is equal to the inclusion $i_U \colon U \to X \times X$. \textit{Topological complexity} of $X$, denoted $TC(X)$, is the least integer $\ell\geq  1$ such that there exists an open cover of $X\times X$ by $\ell$ sets which admit motion planners. (Note that we use the non-reduced version of $TC$ here.)

There exists the following cohomological lower bound for $TC$. (If $R$ is a ring, the \textit{nilpotency} of an ideal $I \subseteq R$, denoted $\mathop{\rm nil}I$, is defined to be the least integer $\ell\geq  0$ such that $I^{\ell+1}=0$.)

\begin{theorem}[{\textnormal{\cite[Theorem 7]{Farber2003a}}}]\label{thm:TC_lower_bound}
For any field $\Bbbk$,
\[ TC(X) > \nilker \!\big[H^*(X;\Bbbk) \otimes H^*(X; \Bbbk) \xrightarrow{\smallsmile} H^*(X;\Bbbk)\big].\]
\end{theorem}

\begin{remark}\label{remark:TC_lower_bound}
Let us comment on how one arrives at this result, as we will invoke the argument later on. Topological complexity can be expressed in terms of sectional category of $\pi_1$, written $\secat (\pi_1)$. If $p \colon E \to B$ is any fibration, then $\secat (p) > \nilker\big[p^* \colon H^*(B;\Bbbk) \to H^*(E;\Bbbk)\big]$ for any field $\Bbbk$ by \cite[Theorem 4]{Schwarz}. Thus
\[ TC(X) = \secat (\pi_1) > \nilker \big[\pi_1^*\colon H^*(X\times X;\Bbbk)\to H^*(PX;\Bbbk)\big].\]
Farber noticed that $\pi_1^*$ coincides with the cup product homomorphism.
\end{remark}

\begin{example}[{\cite[Theorem 8]{Farber2003a}}]
Topological complexity of the sphere $S^n$, $n\geq  1$, is given by:
\[ TC(S^n) = \begin{cases}
2, & \textnormal{$n$ odd,}\\
3, & \textnormal{$n$ even.}
\end{cases}\]
We note that in both cases the lower bound provided by Theorem \ref{thm:TC_lower_bound} turns out to be accurate: $TC(S^n)= \nilker \big[\pi_1^*\colon H^*(S^n\times S^n;\mathbb{Q})\to H^*(PS^n;\mathbb{Q})\big]+1$.
\end{example}

\subsection{Equivariant versions of topological complexity}\label{section:equivariantTC}
There have been several approaches to defining an equivariant counterpart of topological complexity:
\begin{itemize}
\item ``equivariant topological complexity'' (Colman--Grant \cite{Colman2012}),
\item ``strongly equivariant topological complexity'' (Dranishnikov \cite{Dranishnikov2015}),
\item ``invariant topological complexity'' (Lubawski--Marzantowicz \cite{Lubawski2014}).
\end{itemize}
We briefly review their definitions in what follows.

Let $G$ be a topological group. Given a $G$-\hspace{0pt}space $X$, view $PX$ as a $G$-\hspace{0pt}space via the formula
\[ (g\gamma)(-) = g\big(\gamma(-)\big) \textnormal{ for $g \in G$ and $\gamma \in PX$.} \]
The space $X \times X$ can be seen as a $G$-\hspace{0pt}space with the diagonal $G$-\hspace{0pt}action. The map $\pi_1 \colon PX \to X \times X$ then becomes a $G$-\hspace{0pt}fibration. \textit{Equivariant topological complexity} of $X$, $TC_G(X)$, arises as the minimal number $\ell\geq  1$ such that there exists an open $G$-\hspace{0pt}invariant cover of $X\times X$ by $\ell$ sets which admit $G$-\hspace{0pt}equivariant motion planners. \textit{Strongly equivariant topological complexity} of $X$, $TC_G^*(X)$, is defined similarly, only that $X\times X$ is now viewed as a $(G\times G)$-\hspace{0pt}space via the component-wise action, and the open cover in question is required to be $(G\times G)$-\hspace{0pt}invariant. Clearly,
\[ TC(X)\leq  TC_G(X) \leq  TC_G^*(X). \]

The variant of Lubawski and Marzantowicz is a little bit different. Define
\[ PX\times_{X/G} PX= \big\{(\gamma,\delta) \in PX \times PX \suchthat G\gamma(1)=G\delta(0) \big\} \]
and consider it as a $(G\times G)$-\hspace{0pt}space with the component-wise action. The  map $\pi_2 \colon PX \times_{X/G} PX \to X \times X$ given by
\[ \pi_2(\gamma,\delta) = \big(\gamma(0), \delta(1)\big)\textnormal{ for $(\gamma,\delta) \in PX\times_{X/G} PX$} \]
turns out to be a $(G\times G)$-\hspace{0pt}fibration. \textit{Invariant topological complexity} of $X$, $TC^G(X)$, is defined to be the least integer $\ell\geq  1$ such that there exists an open $(G\times G)$-\hspace{0pt}invariant cover $U_1$,~\ldots, $U_\ell$ of $X\times X$ and $(G\times G)$-\hspace{0pt}equivariant maps $s_i \colon U_i \to PX \times_{X/G} PX$ with $\pi_2 \circ s_i = i_{U_i}$, $1\leq  i \leq  \ell$. There is no obvious relationship between $TC(X)$ and $TC^G(X)$, both inequalities can occur. 

All three invariants have a common lower bound in $TC(X^G)$, where $X^G=\{x\in X \suchthat gx=x \textnormal{ for any $g \in G$}\}$ is the fixed point set of $X$. As a result, they can be arbitrarily larger than $TC(X)$. (See \mbox{\cite[Proposition 4.5]{BlaszczykKaluba}} for what we believe to be an interesting family of examples of this sort of behaviour.) In particular, if $X^G$ is disconnected, the invariants are infinite, a phenomenon for which we do not know a satisfactory explanation from the point of view of robotics.

\begin{example}[{\cite[Theorems 3.5 and 3.6]{BlaszczykKaluba}}]
Let $S^n$ be a sphere equipped with a linear $\zmod{p}$-\hspace{0pt}action. Assume that $(S^n)^{\zmod{p}}=S^r$, $0<r<n$. Then:
\begin{enumerate}
\item[\textnormal{(1)}] $TC_{\zmod{p}}(S^n) =\begin{cases}
2, & \textnormal{both $n$ and $r$ are odd,} \\
3, & \textnormal{either $n$ or $r$ is even.}
\end{cases}$
\item[\textnormal{(2)}] $TC^{\zmod{p}}(S^n)=3$, unless $n$ is odd and $r=n-2$, or $n$ is even and $r=n-1$.
\end{enumerate}
The two omitted cases for $TC^{\zmod{p}}$ remain unsettled.
\end{example}

\begin{remark}
We are also aware of the notion of ``groupoid topological complexity'' due to Angel--Colman \cite{AngelColman}. It deals with measuring complexity of the problem of symmetric motion planning in a way which does not force it to be infinite if the fixed point set is disconnected. Interestingly enough, it can be re-interpreted in terms of synchronous motion planning of a group of robots. The bottom line, however, is that it has different properties from the invariant introduced in this paper; in particular, it is bounded from below by Lusternik--Schnirelmann $G$-\hspace{0pt}category (cf. Subsection \ref{subsection:catG_lower_bound}).
\end{remark}

\section{Effective topological complexity}\label{section:effectiveTC}
Let $k\geq  1$ be an integer, $G$ a topological group. Given a $G$-\hspace{0pt}space $X$, write
\[ \P_k(X) = \big\{ (\gamma_1, \ldots, \gamma_k) \in (PX)^k \suchthat  G\gamma_i(1)=G\gamma_{i+1}(0) \text{ for $1\leq  i\leq  k-1$}\big\}. \]
In particular, $\P_1(X)=PX$ and $\P_2(X) = PX \times_{X/G} PX$.
Define the map $\pi_k \colon \P_k(X) \to X \times X$ by
\[ \pi_k(\gamma_1,\ldots,\gamma_k) = \big(\gamma_1(0), \gamma_k(1)\big) \textnormal{ for $(\gamma_1$, \ldots, $\gamma_k) \in \P_k(X)$}. \]
Let us briefly explain that this is a fibration. Following \cite{Lubawski2014}, denote the saturated diagonal of $X$ by $\daleth(X)$, i.e.
\[
\daleth(X) = \big\{(g_1 x, g_2 x) \in X \times X \suchthat \textnormal{$g_1$, $g_2 \in G$ and $x \in X$}\big\}.
\]
Consider the restriction of the fibration $(\pi_1)^k \colon (PX)^k \to (X\times X)^{k}$ with respect to the subspace $X \times \daleth(X)^{k-1} \times X \subseteq (X\times X)^{k}$. The outcome is a fibration $\P_k(X) \to X \times \daleth(X)^{k-1} \times X$. Composing it with the projection onto the first and last coordinates clearly results in $\pi_k$.

\begin{definition}
\begin{enumerate}
\item[(1)] A \textit{$(G,k)$-\hspace{0pt}motion planner} on an open subset $U \subseteq X \times X$ is a section of $\pi_k$ over $U$, i.e. a map $s \colon U \to \P_k(X)$ such that $\pi_k \circ s = i_U$.
\item[(2)] Denote by $\TC{G}{k}(X)$ the least integer $\ell\geq  1$ such that there exists an open cover of $X \times X$ by $\ell$ sets which admit $(G,k)$-\hspace{0pt}motion planners.
\end{enumerate}
\end{definition}

It is clear that $\TC{G}{1}(X) = TC(X)$. Note that we do not require $(G,k)$-\hspace{0pt}motion planners to be equivariant, hence $\TC{G}{2}(X) \leq  TC^G(X)$. Examples of strict inequality are not hard to come by, and we will see many of them later on. 

We will now record a number of properties of $\TC{G}{k}$. Here is what we consider to be the cornerstone of our approach.

\begin{lemma}\label{lemma:monoticity}
The following inequalities hold for any $k\geq  1$ and any subgroup $H\subseteq G$:
\begin{enumerate}
\item[\textnormal{(1)}] $\TC{G}{k}(X)\leq  \TC{H}{k}(X)$,
\item[\textnormal{(2)}] $\TC{G}{k+1}(X)\leq \TC{G}{k}(X)$.
\end{enumerate}
\end{lemma}

\begin{proof}
(1) This is a straightforward consequence of the definition.

(2) Consider a $(G,k)$-motion planner $s=(s_1,\ldots,s_k)$ as a $(G,k+1)$-motion planner by adding a constant path at its end: $s\mapsto\big(s, \textnormal{const}_{s_k(-,-)(1)}\big)$.
\end{proof}

\begin{theorem}\label{thm:G_invariance}
If there exists a $G$-\hspace{0pt}map $f \colon X \to Y$ and a map $g \colon Y \to X$ such that $f\circ g \simeq \id_Y$, then $\TC{G}{k}(Y) \leq \TC{G}{k}(X)$. In particular, if $X$ and $Y$ are $G$-\hspace{0pt}homotopy equivalent, then $\TC{G}{k}(X)=\TC{G}{k}(Y)$.
\end{theorem}
\begin{proof}
The proof follows that of Farber's \cite[Theorem 3]{Farber2003a}. We spell it out here for the convienience of the reader.

Assume that we have a $(G,k)$-\hspace{0pt}motion planner $s=(s_1, \ldots, s_k)$ on an open subset $U \subseteq X \times X$. Let $V = (g\times g)^{-1}(U) \subseteq Y \times Y$ and fix a homotopy $H \colon Y \times [0,1] \to Y$ such that $H(-,0)=\id_Y$ and $H(-,1)=f\circ g$.
Define a map $\tilde{s} = (\tilde{s}_1,\ldots,\tilde{s}_k) \colon V \to \mathcal{P}_k(Y)$ as follows:
\[ \tilde{s}_i(x,y) = \begin{cases}
H(x,-) * f\circ s_1\big(g(x),g(y)\big), & i=1,\\
f\circ s_i\big(g(x),g(y)\big), & 2 \leq  i \leq  k-1,\\
f\circ s_n\big(g(x),g(y)\big) * H^{-1}(y,-), & i=k,
\end{cases}\]
where $*$ denotes concatenation of paths and $H^{-1}(y,-)$ is the path going in the opposite direction to $H(y,-)$.
Clearly, $\tilde{s}_1(x,y)(0)=x$, $\tilde{s}_k(x,y)(1) = y$ and, given that $s$ is a $(G,k)$-\hspace{0pt}motion planner and $f$ is a $G$-\hspace{0pt}map, $\big(\tilde{s}_1(x,y), \ldots, \tilde{s}_k(x,y)\big) \in \mathcal{P}_k(Y)$. This shows that $\tilde{s}$ is a $(G,k)$-\hspace{0pt}motion planner on $V$ and, consequently, $\TC{G}{k}(Y) \leq \TC{G}{k}(X)$.
\end{proof}

For any $G$-\hspace{0pt}space $X$, $\left(\TC{G}{k}(X)\right)_{k=1}^{\infty}$ is a decreasing sequence of $G$-\hspace{0pt}homotopy invariants and since the sequence is bounded from below by $1$, it stabilises at some point.

\begin{definition}
 Let $k_0\geq  1$ be the minimal number such that $\TC{G}{k}(X) = \TC{G}{k+1}(X)$ for any $k\geq  k_0$. Set
\[ \TC{G}{\infty}(X) = \TC{G}{k_0}(X). \]
This is the \textit{effective topological complexity} of $X$.
\end{definition}

One can interpret motion planning in this context as follows. A path output by a $(G,\infty)$-\hspace{0pt}motion planner is typically no longer continuous, but its discontinuities are of prescribed nature -- they are parametrised by symmetries. Whenever a robot follows such a path and runs into a point of discontinuity, it re-interprets its position accordingly within a batch of symmetric positions, and then resumes normal movement.

It is clear that $\TC{G}{\infty}$ satisfies the properties described in Lemma \ref{lemma:monoticity} and Theorem \ref{thm:G_invariance}. An obvious necessary condition for $\TC{G}{\infty}(X)$ to be finite is path-connectedness of the orbit space $X/G$.

\begin{remark}
Effective topological complexity is less than or equal to any other version of equivariant topological complexity discussed in Section \ref{section:equivariantTC}.
\end{remark}

Note that if $X$ is contractible, then, regardless of the nature of a group action, $\TC{G}{\infty}(X)=1$. Recall that a $G$-\hspace{0pt}space $X$ is said to be \textit{$G$-\hspace{0pt}contractible} if the identity $X \to X$ is $G$-\hspace{0pt}homotopic to a $G$-\hspace{0pt}map with values in a single orbit. Clearly any homogeneous $G$-\hspace{0pt}space, in particular, any topological group $G$ considered as a $G$-\hspace{0pt}space via group multiplication, is $G$-\hspace{0pt}contractible.

\begin{proposition}
If a $G$-\hspace{0pt}space $X$ is $G$-\hspace{0pt}contractible, then $\TC{G}{\infty}(X)=1$.
\end{proposition}

\begin{proof}
This follows from the fact that $\TC{G}{\infty}(X)\leq  TC^G(X)$ for any $G$-\hspace{0pt}space~$X$, and the latter equals $1$ whenever $X$ is $G$-\hspace{0pt}contractible, see \cite[Corollary 2.8, Remark~2.9]{BlaszczykKaluba}.
\end{proof}

\section{A lower bound for $\TC{G}{\infty}$}\label{section:effectiveTC_lower_bound}

Effective topological complexity enjoys the following analogue of Theorem \ref{thm:TC_lower_bound}. (In this section we make use of the existence of the transfer map in cohomology without further ado. See \cite[Chapter III, Section 2]{Bredon} for details.)

\begin{theorem}\label{thm:effectiveTC_lower_bound}
Let $G$ be a finite group and $X$ a $G$-\hspace{0pt}CW complex. If $\Bbbk$ is a field of characteristic zero or prime to the order of~$G$, then
\[ \TC{G}{\infty}(X) > \nilker \big[H^*(X/G;\Bbbk) \otimes H^*(X/G;k) \overset{\smallsmile}{\longrightarrow} H^*(X/G;\Bbbk)\big].\]
\end{theorem}

\begin{proof}
We will prove that the inequality in question in fact holds for any $\TC{G}{k}(X)$ with $k \geq  2$.

View $\mathcal{P}_k(X)$ as a $G^k$-\hspace{0pt}space via the component-wise action. The space \mbox{$X \times X$} is naturally a $(G\times G)$-\hspace{0pt}space, but it can be considered as a $G^k$-\hspace{0pt}space by precomposing the action with the projection $G^k \to G\times G$ onto the first and last coordinates. This way $\pi_k \colon \mathcal{P}_k(X) \to X \times X$ becomes a $G^k$-\hspace{0pt}equivariant map, and thus it induces a map $\overline{\pi}_k \colon \mathcal{P}_k(X)/G^k \to X/G \times X/G$ between respective orbit spaces. Let $p \colon P(X/G) \to X/G \times X/G$ denote the usual path fibration and $\eta \colon \mathcal{P}_k(X)/G^k \to P(X/G)$ concatenation of a sequence of $k$ paths in~$X/G$. These maps fit into the following commutative diagram.

\begin{center}
\begin{tikzpicture}
\node (PnX)     at  (0,4) {$\P_k(X)$};
\node (XX)      at  (5,4) {$X\times X$};
\node (PnXG)    at  (0,2) {$\P_k(X)/G^{k}$};
\node (XGXG)    at  (5,2) {$X/G\times X/G$};
\node (PXG)     at  (2.5,0.25) {$P(X/G)$};
\draw [->] (PnX) edge node [above] {$\pi_{k}$} (XX);
\draw [->] (PnX) edge node [left] {$\pi_G'$} (PnXG);
\draw [->] (PnXG) edge node [above] {$\overline{\pi}_{k}$} (XGXG);
\draw [->] (XX) edge node [right] {$\pi_G$} (XGXG);
\draw [->] (PnXG) edge node [above, sloped]{$\eta$} (PXG);
\draw [->] (PXG) edge node [above, sloped]{$p$} (XGXG);
\end{tikzpicture}
\end{center}

Choose a field $\Bbbk$ of characteristic zero or prime to $|G|$ and apply the functor $H^*(-;\Bbbk)$ to the diagram above. Note that if $\alpha \in \ker p^*$, then $\pi_G^*(\alpha) \in \ker \pi_k^*$. Indeed,
\[\pi^*_k\big(\pi_G^*(\alpha)\big) = (\pi_G')^*\big(\bar{\pi}_k^*(\alpha)\big) = \pi_G^*\big(\eta^*(p^*\alpha)\big)=0.\]
Thanks to the choice of coefficients, $\pi_G^*$ is a monomorphism, which implies that if a product of elements $\alpha_1, \ldots, \alpha_\ell\in \ker p^*$ is non-zero, then so is the product of elements $\pi_G^*(\alpha_1), \ldots, \pi_G^*(\alpha_\ell) \in \ker\pi_k^*$. This shows that
\[ \nilker p^*\leq  \nilker\pi_k^*.\]
As mentioned in Remark \ref{remark:TC_lower_bound}, the left-hand side is equal to nilpotency of the cup product homomorphism $H^*(X/G;\Bbbk) \otimes H^*(X/G;\Bbbk) \to H^*(X/G;\Bbbk)$, and the right-hand side constitutes a (sharp) lower bound for $\secat (\pi_k)$. It is, however, clear from the definition that $\TC{G}{k}(X) = \secat (\pi_k)$.
\end{proof}

\begin{corollary}\label{cor:TC=effectiveTC}
Let $G$ be a finite group and $X$
a $G$-\hspace{0pt}CW complex.
If $G$ acts trivially on $H^*(X;\Bbbk)$ and
\[ TC(X)=\nilker\big[H^*(X;\Bbbk) \otimes H^*(X;\Bbbk) \xrightarrow{\smallsmile} H^*(X;\Bbbk)\big]+1\]
for some field $\Bbbk$ of characteristic $0$ or prime to the order of $G$, then $\TC{G}{\infty}(X)=TC(X)$.
\end{corollary}

\begin{proof}
Since $G$ acts trivially on cohomology, $H^*(X/G;\Bbbk) \cong H^*(X;\Bbbk)$ and, consequently:
\begin{align*}
TC(X) &= \nilker\big[H^*(X;\Bbbk) \otimes H^*(X;\Bbbk) \to H^*(X;\Bbbk)\big]+1 \\
&= \nilker\big[H^*(X/G;\Bbbk) \otimes H^*(X/G;\Bbbk) \to H^*(X/G;\Bbbk)\big]+1 \\
&\leq  \TC{G}{\infty}(X) \leq TC(X).
\end{align*}
\end{proof}

\section{Calculations of $\TC{G}{\infty}$}

\subsection{Effective topological complexity of free $G$-spaces}

\begin{theorem}\label{thm:effectiveTC_free_actions}
If $G$ is a finite group and $X$ is a free $G$-\hspace{0pt}space, then $\TC{G}{k}(X)=\TC{G}{k+1}(X)$ for any $k \geq  2$. In particular, $\TC{G}{\infty}(X) = \TC{G}{2}(X)$.
\end{theorem}

\begin{proof}
Let $U \subseteq X\times X$ be an open subset admitting a $(G,k+1)$-\hspace{0pt}motion planner~$s$. We will turn it into a $(G,k)$-\hspace{0pt}motion planner.

Given $(x,y) \in U$, write $s(x,y)=(\gamma_1, \ldots, \gamma_{k+1})$ and set $\Gamma(x,y)$ to be the unique element of $G$ such that $\Gamma(x,y)\gamma_k(1)=\gamma_{k+1}(0)$. (Uniqueness of $\Gamma(x,y)$ is a consequence of freeness of the action.) This gives a continuous function $\Gamma \colon U \to G$. Now define
\[ \tilde{s}(x,y) = \big(\gamma_1, \ldots, \gamma_{k-1}, \Gamma(x,y)\gamma_k * \gamma_{k+1}\big). \]
Clearly, $\tilde{s}(x,y)$ is a ``$(G,k)$-\hspace{0pt}path'' in $X$ between $x$ and $y$. Since both $s$ and $\Gamma$ are continuous, letting $(x,y) \in U$ vary yields a $(G,k)$-\hspace{0pt}motion planner on $U$.
\end{proof}

\begin{lemma}\label{lemma:free_actions_lower_bound}
Let $G$ be a finite group and $X$ a finite-dimensional free $G$-\hspace{0pt}CW complex. If there exists a prime $p$ such that $X$ is not $\operatorname{mod}p$ acyclic and $H^i(X;\zmod{p})$ is finitely generated for all $i\geq 0$, then $\TC{G}{\infty}(X)\geq  2$.
\end{lemma}

\begin{proof}
In view of Theorem \ref{thm:effectiveTC_free_actions}, it suffices to prove that $\TC{G}{2}(X) \geq  2$.
By the freeness assumption, the map $\omega_0 \colon PX \to X/G$ assigning to a path the orbit of its initial point is a fibration. Since $\P_2(X)$ fits into the pullback diagram
\begin{center}
 \begin{tikzpicture}[yscale = 0.8]
  \node (P2M) at (0,2) {$\P_2(X)$};
  \node (PMl) at (0,0) {$PX$};
  \node (PMu) at (2,2) {$PX$};
  \node (MG)  at (2,0) {$X/G$};
  \draw [->] (P2M) edge (PMl);
  \draw [->] (P2M) edge (PMu);
  \draw [->] (PMl) edge node[above] {$\omega_1$} (MG);
  \draw [->] (PMu) edge node[right] {$\omega_0$} (MG);
 \end{tikzpicture}
\end{center}
\noindent the map $\mathcal{P}_2(X) \to PX$ is a fibration whose fibre is homotopy equivalent to $G$. Let $n\geq  1$ be the largest integer such that $H^n(X;\zmod{p})\neq 0$. Applying the Serre spectral sequence with $\zmod{p}$-\hspace{0pt}coefficients shows immediately that \mbox{$H^i(\mathcal{P}_2(X); \zmod{p})=0$} for $i>n$. Therefore $\pi_2^*$ sends any element of $H^{2n}(X\times X; \zmod{p})$ to zero. Since $\TC{G}{2}(X)$ is equal to sectional category of $\pi_2$, the conclusion follows (see Remark~\ref{remark:TC_lower_bound}).
\end{proof}
\noindent It is instructive to compare the last result with Proposition \ref{prop:tc_reflection}.

\subsection{Effective topological complexity of $\zmod{p}$-spheres}
Throughout this section $n \geq  1$ is a fixed integer. All considered spaces are assumed to be $\zmod{p}$-\hspace{0pt}CW complexes or, equivalently, CW complexes with cellular $\zmod{p}$-\hspace{0pt}actions.

\begin{proposition}\label{prop:effectiveTC_zp_spheres}
Let $p>2$ be a prime. Effective topological complexity of a $\zmod{p}$-\hspace{0pt}sphere $S^n$ is given by
\[ \TC{\zmod{p}}{\infty}(S^n) =
\begin{cases}
2, & \textnormal{$n$ is odd,}\\
3, & \textnormal{$n$ is even, $n>0$.}
\end{cases}\]
\end{proposition}

\begin{proof}
Since $p>2$, the action is trivial on rational cohomology, and the conclusion follows from Corollary \ref{cor:TC=effectiveTC}.
\end{proof}

Clearly, if $p=2$ and the action preserves orientation, the conclusion of Proposition \ref{prop:effectiveTC_zp_spheres} still holds. Thus what remains is to deal with orientation-reversing involutions.

Recall that $\daleth(S^n)$ denotes the saturated diagonal of $S^n$. Let $\Delta \colon S^n \to S^n\times S^n$ be the diagonal map and $j \colon \daleth(S^n) \to S^n \times S^n$ the inclusion.

\begin{lemma}\label{lemma:daleth}
Suppose that $\zmod{2}$ acts on $S^n$ with an $r$-\hspace{0pt}dimensional fixed point set,
$0 \leq  r \leq  n-2$. Then:
\begin{enumerate}
\item[\textnormal{(1)}]
 $H^i\big(\daleth(S^n);\zmod{2}\big) \cong \begin{cases}
\zmod{2}, & \textnormal{$i=0$, $r+1$}, \\
\zmod{2}\oplus \zmod{2}, & i=n,
\end{cases}$
\item[\textnormal{(2)}] If $\alpha \in H^n(S^n;\zmod{2})$ is the $\operatorname{mod}2$ fundamental class of $S^n$, then
\[ j^*(\alpha\otimes 1)=j^*(1\otimes \alpha)\neq 0. \]
\end{enumerate}
\end{lemma}
\begin{proof}
(1) Let $\zmod{2}=\{e,g\}$, with $g$ the generator, and define
\[ \daleth_{\kappa}(S^n) = \big\{(\kappa x,x) \in S^n\times S^n \suchthat x\in S^n\big\} \textnormal{ for $\kappa = e,g$}. \]
Observe
that $\daleth(S^n) = \daleth_e(S^n) \cup \daleth_g(S^n)$. Indeed, given $(\kappa_1 x, \kappa_2 x) \in \daleth(S^n)$, set $\tilde{\kappa} = \kappa_1\kappa_2$ and $x'=\kappa_2 x$, so that $(\kappa_1 x, \kappa_2 x) = (\tilde{\kappa}x',x')$. This shows that $\daleth(S^n) = \big\{(\kappa x, x)\in S^n\times S^n \,|\, \textnormal{$\kappa \in \zmod{2}$ and $x \in S^n$} \big\}$, and the latter clearly decomposes as claimed.

Since the action is assumed to be cellular, we can push-forward the CW-structure of $S^n$ to $\daleth_\kappa(S^n)$ via homeomorphism $x\mapsto (\kappa x,x)$. The intersection $\daleth_e(S^n)\cap\daleth_g(S^n)\cong (S^n)^{\zmod{2}}$ is then a common subcomplex and has the $\operatorname{mod}2$ cohomology of a sphere by Smith theory. The conclusion now follows from the Mayer--Vietoris sequence corresponding to the decomposition of $\daleth(S^n)$.

(2) We can identify $H^n\big(\daleth(S^n);\zmod{2}\big)$ with $H^n\big(\daleth_{e}(S^n);\zmod{2}\big)\oplus H^n\big(\daleth_{g}(S^n);\zmod{2}\big)$, as shown above.
The composition \[H^n(S^n\times S^n;\zmod{2})\xrightarrow{j^*} H^n(\daleth(S^n);\zmod{2})\xrightarrow{\cong} H^n\big(\daleth_e(S^n);\zmod{2}\big) \oplus H^n\big(\daleth_g(S^n);\zmod{2}\big)\]
is given by $\xi \mapsto j_e^*(\xi) + j_g^*(\xi)$, where $j_{\kappa} \colon \daleth_{\kappa}(S^n) \to S^n \times S^n$ denotes the inclusion, $\kappa=e,g$.
But $j_{\kappa} \colon \daleth_{\kappa}(S^n) \to S^n\times S^n$ factors as
\[ \daleth_{\kappa}(S^n) \xrightarrow{\operatorname{proj}_2} S^n \xrightarrow{\,\Delta\,} S^n\times S^n \xrightarrow{\kappa\times \id} S^n\times S^n, \]
which induces the same map as $\Delta$ on $\operatorname{mod}2$ cohomology, and the latter is given by $\Delta^*(\alpha\otimes 1) = \Delta^*(1\otimes \alpha)=\alpha$.
\end{proof}
We will denote $j^*(\alpha\otimes 1)=j^*(1\otimes\alpha)=\beta$ from now on.
\pagebreak[2]
\begin{lemma}\label{lemma:orientation_reversing}
If $\zmod{2}$ acts on $S^n$ with an $r$-\hspace{0pt}dimensional fixed point set, \mbox{$0 \leq  r \leq  n-2$}, then $\TC{\zmod{2}}{\infty}(S^n)\geq  2$.
\end{lemma}
During the course of the proof we will use the following subscript convention:
if $\ell\geq  1$ and $\xi \in H^*(X;\zmod{2})$, then
\[ \xi_i = 1\otimes \cdots \otimes 1 \otimes\underset{\textnormal{$i$-th}}\xi\otimes 1 \otimes\cdots\otimes 1\in H^*(X;\zmod{2})^{\otimes \ell} \textnormal{ for $1\leq  i \leq  \ell$.}\]
\begin{proof}
As explained right after the proof of Proposition  \ref{prop:effectiveTC_zp_spheres}, the conclusion is \emph{a fortiori} true for orientation-preserving actions, hence we restrict our attention to orientation-reversing involutions. Furthermore, since any action on $S^2$ is equivalent to an orthogonal one, we can assume that $n\geq  3$. Indeed, if an orthogonal action on $S^2$ has a $0$-\hspace{0pt}dimensional fixed point set, then it clearly preserves orientation.

Let $k\geq  2$. Consider the following commutative diagram.
\begin{center}
\begin{tikzpicture}
 \node (OSnUp) at (-5,2) {$(\Omega S^n)^k$};
 \node (OSnDown) at (-5,0) {$(\Omega S^n)^k$};
 \node (OSn) at (-5,-2) {$\Omega S^n$};
 \node (PkSn) at (-2.5,2) {$\P_k(S^n)$};
 \node (PSnk) at (-2.5,0) {$(PS^n)^k$};
 \node (PSn) at (-2.5,-2) {$PS^n$};
  \node (Sn-Daleth-Sn) at (1.5,2) {$S^n\times \daleth(S^n)^{k-1}\times S^n$};
  \node (SnSnk) at (1.5,0) {$(S^n\times S^n)^k$};
\node (SnSn)  at (1.5,-2) {$S^n\times S^n$};
  \draw [->] (OSnUp) edge (PkSn);
  \draw [->] (OSnUp) edge node [left] {$\id$} (OSnDown);
  \draw [->] (OSnDown) edge (PSnk);
  \draw [->] (PkSn) edge node [left] {$f$}(PSnk);
  \draw [->] (PkSn) edge (Sn-Daleth-Sn);
  \draw [->] (Sn-Daleth-Sn) edge node [right] {$\bar{\jmath}=\id\times j^{k-1}\times \id$} (SnSnk);
  \draw [->] (PSnk) edge (SnSnk);
  \draw [->] (PSnk) edge (SnSn);
  \draw [->] (SnSnk) edge node [right] {$\operatorname{proj}_{1, 2k}$}(SnSn);
  \draw [->] (OSn) edge (PSn);
  \draw [->] (PSn)edge (SnSn);
 \end{tikzpicture}
\end{center}
The top row arises as a restriction of the middle one, which is the $k$-\hspace{0pt}fold product of the path space fibration $\pi_1$, which in turn is represented in the bottom row. The map $f \colon \P_k(S^n)\to (PS^n)^k$ is the inclusion, and the unmarked oblique map is given by $(\gamma_1, \ldots, \gamma_k) \mapsto \big(\gamma_1(0),\gamma_k(1)\big)$ for any $(\gamma_1,\ldots, \gamma_k) \in (PS^n)^k$.

Note that the composition $\P_k(S^n) \to (PS^n)^k \to S^n \times S^n$ coincides with~$\pi_k$, thus as soon as we understand the kernel of $f^*$, we will be able to infer information about the kernel of $\pi_k^*$. In order to do so, we will analyse and compare Serre spectral sequences corresponding to fibrations in the diagram above.

Let $\big(\widetilde{E}_{*}^{*,*},\tilde{d}_*\big)$, $\big(\bar{E}_{*}^{*,*},\bar{d}_*\big)$ and $\big(E_*^{*,*},d_*\big)$ denote the $\operatorname{mod}2$ cohomology Serre spectral sequences corresponding to fibrations in the bottom, middle, and top rows of the diagram, respectively. As these fibrations either have simply connected bases or are pull-backs of such, they are oriented, hence spectral sequences in question have untwisted coefficients.

Recall that
\[ H^\ell(\Omega S^n;\zmod{2})\cong
\begin{cases}
\zmod{2}, & \ell\equiv 0\textnormal{\,mod\,$(n-1)$},\\
0, & \textnormal{otherwise.}
\end{cases}
\]
Therefore, provided that $r\geq  1$, each of these sequences has only one non-trivial differential targeting a single non-zero entry on the $n$-\hspace{0pt}th diagonal; denote those by $\tilde{d}$, $\bar{d}$ and $d$, respectively. This justifies the isomorphisms drawn in the diagram below.
\begin{center}
\begin{tikzpicture}[yscale = 0.8]
\node (HnSnSn) at (0,4) {$H^n(S^n\times S^n;\zmod{2})$};
\node (HnPSnk) at (3,2) {$H^n\big((PS^n)^k;\zmod{2}\big)$};
\node [right = -0.1cm of HnPSnk] (HnPSnkiso) {$\cong H^n\big((S^n\times S^n)^{k};\zmod{2}\big)\!\big/\!\im \bar{d}$};
\node (HnPkSn) at (0,0) {$H^n\big(\P_k(S^n);\zmod{2}\big)$};
\node [right = -0.1cm of HnPkSn] (HnPkSniso) {$\cong H^n\big(S^n\times \daleth(S^n)^{k-1}\times S^n;\zmod{2}\big)\!\big/\!\im d$};
\draw [->] (HnPSnk) edge node [above, sloped] {$f^*$} (HnPkSn);
\draw [->] (HnSnSn) edge (HnPSnk);
\draw [->] (HnSnSn) edge node [left] {$\pi_k^*$} (HnPkSn);
\end{tikzpicture}
\end{center}
\noindent If $r=0$, there possibly is one more non-zero entry on the $n$-\hspace{0pt}th diagonal for $\big(E_*^{*,*},d_*\big)$, but this will not be an issue, as $H^n\big(S^n\times \daleth(S^n)^{k-1}\times S^n;\zmod{2}\big)/\im d$ still appears as a direct summand of $H^n\big(\P_k(S^n);\zmod{2}\big)$.

In both cases our task is to identify the relevant differentials. Let $\gamma \in H^{n-1}(\Omega S^n;\zmod{2})$ be the generator. Consider $\big(\widetilde{E}_{*}^{*,*},\tilde{d}_*\big)$ first. The differential $\tilde{d} \colon \widetilde{E}_n^{0,n-1} \to \widetilde{E}_n^{n,0}$ is given by $\tilde{d}(\gamma) = \alpha\otimes 1 + 1\otimes\alpha$. Indeed, it is well-known that $\pi_1^*$ can be expressed as the edge homomorphism
\[ H^n(S^n\times S^n;\zmod{2}) \cong \widetilde{E}_n^{n,0} \to \widetilde{E}_\infty^{n,0} \cong H^n(S^n\times S^n; \zmod{2})\!\big/\!\im\tilde{d} \cong H^n(PS^n;\zmod{2}), \]
thus its kernel is precisely the image of $\tilde{d}$. But $\pi_1^*$ is also easily seen to coincide with $\Delta^*$. Therefore,
\[ \pi_1^*(\alpha\otimes 1 + 1\otimes\alpha) =
\Delta^*(\alpha\otimes 1)+\Delta^*(1\otimes\alpha) = 2\alpha =0 \in H^n(PS^n;\zmod{2}),\]
so that $\alpha\otimes 1 + 1\otimes \alpha \in \im \tilde{d}$.
Consequently, identifying $H^*\big((S^n\times S^n)^k;\zmod{2}\big)$ with $H^*\big((S^n)^{2k};\zmod{2}\big)$, we see that
\begin{multline*}
\bar{d} \colon \bar{E}_n^{0,n-1} \cong H^0\Big((S^n)^{2k};H^{n-1}\big((\Omega S^n)^k;\zmod{2}\big)\Big)
\to\\
\to H^{n}\Big((S^n)^{2k};H^0\big((\Omega S^n)^k;\zmod{2}\big)\Big)\cong\bar{E}_n^{n,0}
\end{multline*}
is given by $\bar{d}(\gamma_i) = \alpha_{2i-1} + \alpha_{2i}$
for $1\leq  i\leq  k$.

Now recall that the map
$(\id,f,\bar{\jmath})$ of fibrations induces a map $f_n \colon \bar{E}_n^{*,*} \to E_n^{*,*}$
which satisfies $d\circ f_n=f_n\circ\bar{d}$ by naturality of spectral sequences.
\begin{center}
\begin{tikzpicture}
    [axis-label/.style={font=\scriptsize},
    helpline/.style={thin,blue!20}]
    \coordinate (P1) at (-7,1.5); 
    \coordinate (P2) at (9,1.5); 
    \coordinate (A1) at (0,3); 
    \coordinate (A2) at (0,-3); 
    \coordinate (A3) at ($(P1)!.6!(A2)$); 
    \coordinate (A4) at ($(P1)!.6!(A1)$);
    \coordinate (A7) at ($(P2)!.6!(A2)$);
    \coordinate (A8) at ($(P2)!.6!(A1)$);
    \coordinate (A5) at
      (intersection cs: first line={(A8) -- (P1)},
                second line={(A4) -- (P2)});
    \coordinate (A6) at
      (intersection cs: first line={(A7) -- (P1)},
                second line={(A3) -- (P2)});
    \def \n {7.5}
    \def \k {5}
    \begin{scope}[] 
    \draw (A1) node [above left] {$E_{n}^{*,*}$};
    \begin{scope}
    \begin{scope}
        \foreach \x in {0.5,1.5,...,\n}
        {
        \draw [helpline] ($(A3)!\x/(\n+0.5)!(A2)$) -- ($(A4)!\x/(\n+0.5)!(A1)$);
        \draw [helpline] ($(A3)!\x/(\n+0.5)!(A4)$) -- ($(A2)!\x/(\n+0.5)!(A1)$);
        }
        \coordinate (e1) at ($(A2)!1/(\n+0.5)!(A1)$);
        \coordinate (e2) at ($(A2)!0/(\n+0.5)!(A1)$);
        \coordinate (e3) at
        (intersection cs: first line = {(e1)--(P1)},
                        second line = {(A3)--(A4)});
        \coordinate (e4) at
        (intersection cs: first line = {(e2)--(P1)},
                        second line = {(A3)--(A4)});
        \fill[gray!50, opacity=0.4] (e1)--(e2)--(e4)--(e3)--cycle;

        \coordinate (E1) at ($(A2)!\k/(\n+0.5)!(A1)$);
        \coordinate (E2) at ($(A2)!\k/(\n+0.5) + 1 /(\n+0.5) !(A1)$);
        \coordinate (E3) at
        (intersection cs: first line = {(E1)--(P1)},
                        second line = {(A3)--(A4)});
        \coordinate (E4) at
        (intersection cs: first line = {(E2)--(P1)},
                        second line = {(A3)--(A4)});
        \fill[gray!50,opacity=0.4] (E1)--(E2)--(E4)--(E3)--cycle;
    \end{scope}
    \node (E-0row) [axis-label, left, align=right] at ($(e3)!.5!(e4)$) {$0$};
    \node (E-n-1row) [axis-label, left, align=right] at ($(E3)!.5!(E4)$) {$n-1$};
    \draw[->] (A3) -- ($(A3)!1.05!(A2)$);
    \draw[->] (A3) -- ($(A3)!1.05!(A4)$);
    \end{scope}
    \node (E-diff-source) at
        (intersection cs:
        first line = {($(A2)!\k/(\n+0.5) + 0.5/(\n+0.5)!(A1)$)--
                      ($(A3)!\k/(\n+0.5) + 0.5/(\n+0.5)!(A4)$)},
       second line = {($(A3)!0.5/(\n+0.5)!(A2)$)--
                      ($(A4)!0.5/(\n+0.5)!(A1)$)})
        {$\bullet$};
    \node (E-diff-destination) at
        (intersection cs:
        first line = {($(A3)!\k/(\n+0.5) + 1.5/(\n+0.5)!(A2)$)--
                      ($(A4)!\k/(\n+0.5) + 1.5/(\n+0.5)!(A1)$)},
       second line = {($(A3)!0.5/(\n+0.5)!(A4)$)--
                      ($(A2)!0.5/(\n+0.5)!(A1)$)})
        {$\bullet$};
    \draw[->] (E-diff-source) edge node [axis-label,above,sloped] {$d$} (E-diff-destination);
    \end{scope}
    \begin{scope}[opacity=.7] 
        \draw (A8) node [above left] {$\bar{E}_{n}^{*,*}$};
        \draw[->] (A6) -- ($(A6)!1.05!(A5)$);
        \draw[->] (A6) -- ($(A6)!1.05!(A7)$);
        \begin{scope}
            \foreach \x in {0.5,1.5,...,\n}
            {
            \draw [helpline,opacity=0.7] ($(A6)!\x/(\n+0.5)!(A7)$) -- ($(A5)!\x/(\n+0.5)!(A8)$);
            \draw [helpline,opacity=0.7] ($(A6)!\x/(\n+0.5)!(A5)$) -- ($(A7)!\x/(\n+0.5)!(A8)$);
            }
            \coordinate (eb1) at
            (intersection cs: first line = {(e1)--(P2)},
                            second line = {(A7)--(A8)});
            \coordinate (eb2) at
            (intersection cs: first line = {(e2)--(P2)},
                            second line = {(A7)--(A8)});
            \coordinate (eb3) at
            (intersection cs: first line = {(e3)--(P2)},
                            second line = {(A6)--(A5)});
            \coordinate (eb4) at
            (intersection cs: first line = {(e4)--(P2)},
                            second line = {(A6)--(A5)});
            \fill[gray!50,opacity=0.4] (eb1)--(eb2)--(eb4)--(eb3)--cycle;

            \coordinate (Eb1) at
            (intersection cs: first line={(E1)--(P2)},
                            second line={(A7)--(A8)});
            \coordinate (Eb2) at
            (intersection cs: first line={(E2)--(P2)},
                            second line={(A7)--(A8)});
            \coordinate (Eb3) at
            (intersection cs: first line={(E3)--(P2)},
                            second line={(A6)--(A5)});
            \coordinate (Eb4) at
            (intersection cs: first line={(E4)--(P2)},
                            second line={(A6)--(A5)});
            \fill[gray!50,opacity=0.4] (Eb1)--(Eb2)--(Eb4)--(Eb3)--cycle;
        \end{scope}
        \node (Eb-0row) [axis-label, right] at ($(eb1)!.5!(eb2)$) {$0$};
        \node (Eb-n-1row) [axis-label, right] at ($(Eb1)!.5!(Eb2)$) {$n-1$};
        \node (Eb-diff-source) at
            (intersection cs:
            first line = {(E-diff-source)--(P2)},
            second line = {($(A6)!0.5/(\n+0.5)!(A7)$)--
                        ($(A5)!0.5/(\n+0.5)!(A8)$)})
            {$\bullet$};
        \node (Eb-diff-destination) at
            (intersection cs:
            first line = {($(A6)!\k/(\n+0.5) + 1.5/(\n+0.5)!(A7)$)--
                        ($(A5)!\k/(\n+0.5) + 1.5/(\n+0.5)!(A8)$)},
        second line = {($(A6)!0.5/(\n+0.5)!(A5)$)--
                        ($(A7)!0.5/(\n+0.5)!(A8)$)})
            {$\bullet$};
        \draw[->] (Eb-diff-source) edge node [axis-label,above,sloped] {$\bar{d}$} (Eb-diff-destination);
    \end{scope}
    \draw[->,bend right=25] (Eb-diff-source) edge node [axis-label,above,sloped] {$f_n=\id$} (E-diff-source);
    \draw[->,bend right=15] (Eb-diff-destination) edge node [axis-label,above,sloped] {$f_n=\bar{\jmath}$} (E-diff-destination);
\end{tikzpicture}
\end{center}

\noindent Observe that $f_n \colon \bar{E}_n^{0,n-1} \to E_n^{0,n-1}$
is the identity and $f_n \colon \bar{E}_n^{n,0}\to E_n^{n,0}$ is induced by the inclusion $\bar{\jmath}$. Therefore $d=\bar{\jmath}^*\circ\bar{d}$ and, with slight abuse of the subscript notation, we can write that
\[
d(\gamma_i) =
\bar{\jmath}^*(\alpha_{2i-1}+\alpha_{2i}) =
\begin{cases}
\alpha_1 + \beta_2, & i=1, \\
\beta_i + \beta_{i+1}, & 2 \leq  i \leq  k-1, \\
\beta_k + \alpha_{k+1}, & i=k,
\end{cases}\]
where we identified
\[H^*\big(S^n \times \daleth(S^n)^{k-1}\times S^n;\zmod{2}\big) \cong H^*(S^n;\zmod{2})\otimes H^*\big(\daleth(S^n);\zmod{2}\big)^{\otimes (k-1)} \otimes H^*(S^n;\zmod{2}). \]

Since our aim is to understand $\pi_k^*$, we will be interested in behaviour of $f^*$ on the elements $\alpha_1+\im\bar{d}$, $\alpha_{2k} +\im\bar{d} \in H^n\big((PS^n)^k;\zmod{2}\big)$, i.e. those which come from $H^n(S^n\times S^n;\zmod{2})$ by means of the oblique map in the diagram. Obviously, we have:
\begin{align*}
f^*(\alpha_1 + \im\bar{d}\,) &= \alpha_1 + \im d, \\
f^*(\alpha_{2k} + \im\bar{d}\,) &= \alpha_{k+1} + \im d.
\end{align*}
But, per our computations,
$\sum_{i=1}^k d(\gamma_i) = \alpha_1 + \alpha_{k+1}$,
thus $(\alpha_1 + \alpha_{2k}) + \im\bar{d} \in \ker f^*$ and the conlusion follows (see Remark \ref{remark:TC_lower_bound}).
\end{proof}

Given non-antipodal points $x$, $y \in S^n$, let $s'(x,y)$ denote the unique shortest arc connecting $x$ and $y$ traversed with constant velocity. This notation will be useful in proofs of Propositions \ref{prop:orientation_reversing} and \ref{prop:tc_reflection}.

\begin{proposition}\label{prop:orientation_reversing}
Suppose that $\zmod{2}$ acts on $S^n$ reversing orientation and with an $r$-dimensional fixed point set, $0 \leq r\leq n-2$. If: 
\begin{enumerate}
\item[\textnormal{(1)}] $n$ is odd, or
\item[\textnormal{(2)}] $n$ is even and the action is linear,
\end{enumerate}
then $\TC{\zmod{2}}{\infty}(S^n)=2$.
\end{proposition}

\begin{proof} 
(1) This follows immediately from Lemmas \ref{lemma:monoticity} and \ref{lemma:orientation_reversing}, as $TC(S^n)=2$ for odd $n$.

(2) As a consequence of Lemma \ref{lemma:orientation_reversing}, we have $\TC{\zmod{2}}{\infty}(S^n)>1$. We will construct a two-fold open cover $U_1$, $U_2$ of $S^n \times S^n$ by sets that admit $(\zmod{2},2)$-\hspace{0pt}motion planners.

For an auxiliary step, let $\tilde{r}=n-r-1$. The action in question is equivalent to the one given by the involution
\[ g \colon (x_0, \ldots, x_n) \mapsto (-x_0, \ldots, -x_{\tilde{r}}, x_{\tilde{r}+1}. \ldots, x_n),\]
Since it reverses orientation, $\tilde{r}$ is even and the fixed point set $(S^n)^{\zmod{2}} = \big\{(x_0,\ldots,x_n)\in S^n \,|\, x_i=0 \text{ for $i\leq  \tilde{r}$}\big\}$ is odd-dimensional. Define a homeomorphism $\tau \colon S^n \to S^n$ by
\[ \tau(x_0,\ldots, x_n) = (-x_0,\ldots, -x_{\tilde{r}}, -x_{\tilde{r}+2}, x_{\tilde{r}+1},\ldots, -x_n, x_{n-1}).\]
We are now in position to produce the said cover of $S^n \times S^n$. Set:
\begin{align*}
U_1 &= \big\{(x,y) \in S^n\times S^n \suchthat y\neq -x\big\}, \\
U_2 &= \big\{(x,y) \in S^n \times S^n \suchthat y\neq-\tau(x)\big\}.
\end{align*}
It can be easily verified that $(x,-x)\in U_2$ for any $x\in S^n$, hence these sets in fact form an open cover of $S^n \times S^n$. Furthermore, $s'$ is a motion planner on~$U_1$ in the classical sense. As explained in Lemma \ref{lemma:monoticity}, we can consider $s'$ as a $(\zmod{2},2)$-\hspace{0pt}motion planner.

Now define $s_2 \colon U_2 \to \P_2(S^n)$ as follows:
\[ s_2(x,y) = \big(\!\const_x, s'(gx, \tau(x)) * s'(\tau(x),y)\big) \textnormal{ for $(x,y) \in U_2$}.\]
Since $\tau$ is chosen so that $\tau(x)\neq -gx$ for any $x \in S^n$, the map $s_2$ is well-defined. It is straightforward to see that it is also continuous, thus it constitutes a $(\zmod{2},2)$-\hspace{0pt}planner on $U_2$.
\end{proof}

In Propositions \ref{prop:tc_reflection} and \ref{prop:free_involutions}  we deal with the two excluded cases, i.e. co\-di\-men\-sion-one fixed point set and free actions, respectively. Their proofs are essentially easier versions of the argument used for justifying Proposition \ref{prop:orientation_reversing}. Also note that $\zmod{2}$ can be readily replaced with any finite group in Proposition \ref{prop:free_involutions}.

\begin{proposition}\label{prop:tc_reflection}
If $\zmod{2}$ acts on $S^n$ by reflection interchanging the hemispheres \textup{(}i.e. the action is linear and $r=n-1$\textup{)}, then  $TC^{\zmod{2},\infty}(S^n)=1$.
\end{proposition}

\begin{proof}
Let $\Gamma(x)$ be the trivial element of $\zmod{2}$ if $x \in S^n_+$ and its generator otherwise. Define a $(\zmod{2},3)$-\hspace{0pt}motion planner $s \colon S^n \times S^n \to \mathcal{P}_3(S^n)$ by setting
\[ s(x,y) = \big(\!\const_x, s'(\Gamma(x)x, N) * s'(N,\Gamma(y)y), \const_y\!\big) \textnormal{ for $x$, $y \in S^n$,}\]
where $N \in S^n$ denotes the north pole.
\end{proof}

The argument above can be easily modified to yield:

\begin{corollary}
Let $X$ be a topological space, $A \subseteq X$ its closed subspace. Consider the adjuntion space $X\cup_A X$ equipped with a reflection which interchanges the two copies of $X$. Then $\TC{\zmod{2}}{\infty}(X\cup_A X)\leq  TC(X)$.
In particular, the suspension $\Sigma X$ of a space $X$, equipped with a reflection interchanging the two cones, has $TC^{\zmod{2},\infty}(\Sigma X)=1$.
\end{corollary}

\begin{proposition}\label{prop:free_involutions}
If $\zmod{2}$ acts freely on $S^n$, then $TC^{\zmod{2}, \infty}(S^n)= 2$.
\end{proposition}

\begin{proof}
In view of Theorem \ref{thm:effectiveTC_free_actions}, we can restrict attention to $(\zmod{2},2)$-\hspace{0pt}motion planners. Lemma \ref{lemma:free_actions_lower_bound} tells us that $\TC{G}{2}(S^n)>1$.
Set:
\begin{align*}
U_1 &= \big\{(x,y) \in S^n \times S^n \suchthat y\neq-x\big\},\\
U_2 &= \big\{(x,y) \in S^n \times S^n \suchthat y\neq-gx \big\}.
\end{align*}
Clearly, $U_1$ and $U_2$ form an open cover of $S^n \times S^n$. A $(\zmod{2},2)$-\hspace{0pt}motion planner on~$U_1$ arises in exactly the same way as in Proposition \ref{prop:orientation_reversing}. Define $s_2 \colon U_2 \to \P_2(S^n)$ as follows:
\[ s_2(x,y) = \big(\!\const_x, s'(gx,y)\big) \textnormal{ for $(x,y) \in U_2$}.\]
\end{proof}

Let us summarise the content of this section below. (We take $r=-1$ to mean that the fixed point set is empty, i.e. the action is free.)

\begin{corollary}
Let $p$ be a prime. Suppose that $\zmod{p}$ acts on $S^n$ with an $r$-\hspace{0pt}dimensional fixed point set, $-1\leq  r \leq  n-1$.
\begin{itemize}
\item If $p>2$, then $\TC{\zmod{p}}{\infty}(S^n) = TC(S^n) = \begin{cases}
2, & \textnormal{$n$ is odd,}\\
3, & \textnormal{$n$ is even, $n>0$.}
\end{cases}$
\item If $p=2$, then $\TC{\zmod{p}}{\infty}(S^n)$ depends on $r$ as follows:

\begin{center}
\begin{minipage}{0.72\textwidth}
\begin{tabular} {@{}ccc@{}}\toprule
& \textsc{O. preserving} & \textsc{O. reversing}\\ \midrule
$r = -1$ & \multicolumn{2}{c}{$\;\;2$}\\ \cmidrule(rl){1-3}
\multirow{4}{*}{$0\leq  r \leq  n-2$}
& \multicolumn{2}{c@{}}{\centering$2$}\\
& \multicolumn{2}{c@{}}{\footnotesize{for $n$ -- odd}}\\\cmidrule(lr){2-3}
& $3$ & $2$\\
& \footnotesize{for $n$ -- even} & \footnotesize{for $n$ -- even, if linear}\\ \cmidrule(rl){1-3}
\multirow{2}{*}{$r=n-1$}& --- & $1$ \\
& \footnotesize{not possible}&\footnotesize{if linear}\\\bottomrule
\end{tabular}
\end{minipage}
\end{center}

\end{itemize}
\end{corollary}

\section{Product inequality}

\begin{theorem}\label{thm:product_inequality}
Let $X$ be a $G$-\hspace{0pt}space, $Y$ an $H$-\hspace{0pt}space, both path-connected and metric. Consider the product $X\times Y$ with the component-wise $(G\times H)$-\hspace{0pt}action. Then
\[ \TC{G\times H}{\infty}(X\times Y) \leq  \TC{G}{\infty}(X) + \TC{H}{\infty}(Y)-1. \]
\end{theorem}

\begin{proof}
A $(G,k)$-\hspace{0pt}motion planner on $U$ and a $(H,k)$-\hspace{0pt}motion planner on $V$ combine to give a $(G\times H, k)$-\hspace{0pt}motion planner on $U\times V$, and the proof is exactly the same as that of Farber's \cite[Theorem 11]{Farber2003a}.
\end{proof}

Note that an analogue of Theorem \ref{thm:product_inequality} for $G=H$ and the component-wise action replaced with the diagonal one is no longer true.

\begin{example}
Consider an involution on $S^1$ which interchanges the two hemispheres and leaves $S^0=\big\{(-1,0), (1,0)\big\}$ fixed. In view of Proposition \ref{prop:tc_reflection}, this has $\TC{\zmod{2}}{\infty}(S^1)=1$. If the said analogue held, then the corresponding diagonal involution on $S^1\times S^1$ would also have $\TC{\zmod{2}}{\infty}(S^1\times S^1)=1$. However, it is easy to see that its orbit space is homeomorphic to $S^2$. Thus Theorem \ref{thm:effectiveTC_lower_bound} implies that $\TC{\zmod{2}}{\infty}(S^1\times S^1)>2$ and, consequently, $\TC{\zmod{2}}{\infty}(S^1\times S^1) = TC(S^1\times S^1)=3$.
\end{example}

\section{Closing comments}

\subsection{Relation of $\TC{G}{\infty}$ and $\operatorname{cat}_G$}\label{subsection:catG_lower_bound}

Topological complexity is well-known to be related to Lusternik--Schnirelmann category:
\begin{theorem}[\textnormal{{\cite[Theorem 5]{Farber2003a}}}]\label{eq:TC_vs_cat}
If $X$ is path-connected and paracompact, then
\[ \cat(X) \leq  TC(X) \leq  2\cat(X)-1. \]
\end{theorem}
\noindent Thus it is natural to ask whether $\TC{G}{\infty}$ is somehow related to Lusternik--Schnirelmann $G$-\hspace{0pt}category. (Recall that $\operatorname{cat}_G(X)$ is defined as follows. A $G$-\hspace{0pt}invariant subset $U\subseteq X$ is called \textit{$G$-\hspace{0pt}categorical} if the inclusion $U \to X$ is $G$-\hspace{0pt}homotopic to a $G$-\hspace{0pt}map with values in a single orbit. Then $\operatorname{cat}_G(X)$ is defined to be the least integer $\ell\geq  1$ such that there exists an open cover of $X$ by $\ell$ sets which are $G$-\hspace{0pt}categorical.) In fact, both $TC_G$ and $TC^G$ enjoy a $\operatorname{cat}_G$ analogue of Theorem \ref{eq:TC_vs_cat}, at least for spaces with non-empty fixed point sets.

Let us explain that we cannot possibly hope for a lower bound for $\TC{G}{\infty}$ in terms of $\operatorname{cat}_G$. In fact, the following examples show that if $\mathcal{TC}$ were a \textit{somehow} defined homotopy invariant with the property that $\mathcal{TC}(X)\leq  TC(X)$, then it likewise would not be possible to obtain such a lower bound. In other words, this phenomenon is not governed by the choice of particular definition, but rather by adopted ``philosophy''.

\begin{itemize}
\item  Let $G$ be a non-trivial discrete group. The universal space $EG$ is contractible, hence $TC(EG)=1$ and, consequently, $\mathcal{TC}(EG)=1$. But $EG$ is a free $G$-\hspace{0pt}space, and so
\[\operatorname{cat}_G(X)=\operatorname{cat}(EG/G)=\operatorname{cat}(BG),\] 
which is equal to the cohomological dimension of $G$, see \cite{EilenbergGanea}. In particular, the difference can be arbitrarily large.

Since $EG$ is not $G$-\hspace{0pt}contractible (a necessary condition for $G$-\hspace{0pt}contractibility of a $G$-\hspace{0pt}space $X$ is contractibility of the orbit space $X/G$), this also shows that we cannot have the property that $\mathcal{TC}(X)=1$ if and only if $X$ is $G$-\hspace{0pt}contractible.

\item Consider the antipodal action on $S^n$. Then \[ \operatorname{cat}_{\zmod{2}}(S^n)=\operatorname{cat}(\mathbb{R}P^n)=n+1. \]
On the other hand, $TC(S^n)\leq  3$. Thus we cannot have $\cat_G \leq  \mathcal{TC}$ even in the realm of actions on closed manifolds.

\item The outlook is the same if we turn attention to actions with fixed points. In \cite[Proposition~4.5]{BlaszczykKaluba} we have constructed an example of a $\zmod{p}$-\hspace{0pt}action on~$S^n$, $n\geq  5$, with an essential homology $(n-2)$-\hspace{0pt}sphere $\Sigma$ as the fixed point set, and this has $\operatorname{cat}_{\zmod{p}}(S^n)\geq  \operatorname{cat}(\Sigma)\geq  n-1$.
\end{itemize}

\begin{remark}
We point out that the first two examples above also show that $\TC{G}{\infty}(X)$ does not coincide with $TC(X/G)$ in general, not even for free actions.
\end{remark}

\subsection{Open problems}

Let us conclude the paper with several natural problems which we do not know the answers to at this point.

\begin{itemize}
\item Identify the fibre of fibration $\pi_k$. (This was raised by J. Gonzalez during the \textit{Workshop on $TC$ and Related Topics}, held at Oberwolfach in March 2016.)
\item Characterize the class of spaces with $\TC{G}{\infty}(X)= 1$. A reasonable guess seems to be that this happens if and only if the orbit space $X/G$ is contractible, provided that $G$ is a finite group and $X$ is a finite-dimensional $G$-CW complex. Under these assumptions contractibility of~$X$ implies contractibility of $X/G$, see \cite[Theorem 6.15]{tomDieck}. This is clearly required in the ``only if'' direction, as contractible spaces have effective $TC$ equal to~$1$. An easy counterexample where this fails without finite-dimensionality assumption is $EG$.
\item Determine what sort of sequences can arise as $\big(\TC{G}{k}(X)\big)_{k=1}^{\infty}$. They certainly need to be non-increasing and bounded from below by $1$, but perhaps there are other restrictions, presumably related to the nature of the action (cf. Theorem \ref{thm:effectiveTC_free_actions}). In particular, it would be interesting to know whether such a sequence can be arbitrarily long before it stabilizes and how big the difference between two consecutive elements of a sequence can be.
\end{itemize}

\noindent\textbf{Acknowledgements.} We wish to express gratitute to H. Colman, M. Grant and M. Farber for suggestions and comments on some of our ideas during the XXIst Oporto Meeting on Geometry, Topology and Physics.

\small
\bibliography{Equivariant-TC}

\providecommand{\bysame}{\leavevmode\hbox to3em{\hrulefill}\thinspace}
\providecommand{\MR}{\relax\ifhmode\unskip\space\fi MR }
\providecommand{\MRhref}[2]{%
  \href{http://www.ams.org/mathscinet-getitem?mr=#1}{#2}
}
\providecommand{\href}[2]{#2}
\begin{thebibliography}{10}

\bibitem{AngelColman}
Andres Angel and Hellen Colman, \emph{{Groupoid topological complexity}},
  \textnormal{talk at the XXIst Oporto Meeting on Geometry, Topology and
  Physics} (2015).

\bibitem{BlaszczykKaluba}
Zbigniew {\noop{Blaszczyk}{Błaszczyk}} and Marek Kaluba, \emph{{On equivariant
  topological complexity of smooth $\zmod{p}$-spheres}}, arXiv:1501.07724
  (2015).

\bibitem{Bredon}
Glen~E. Bredon, \emph{{Introduction to Compact Transformation Groups}},
  Academic Press, 1972.

\bibitem{Colman2012}
Hellen Colman and Mark Grant, \emph{{Equivariant topological complexity}},
  Algebr. Geom. Topol. \textbf{12} (2012), 2299--2316.

\bibitem{tomDieck}
Tammo {\noop{Dieck}{tom Dieck}}, \emph{{Transformation Groups}}, de Gruyter,
  1987.

\bibitem{Dranishnikov2015}
Alexander Dranishnikov, \emph{{On topological complexity of twisted products}},
  Topology Appl. \textbf{179} (2015), 74--80.

\bibitem{EilenbergGanea}
Samuel Eilenberg and Tudor Ganea, \emph{{On the Lusternik--Schnirelmann
  category of abstract groups}}, Ann. of Math. \textbf{65} (1957), 517--518.

\bibitem{Farber2003a}
Michael Farber, \emph{{Topological complexity of motion planning}}, Discret.
  Comput. Geom. \textbf{29} (2003), 211--221.

\bibitem{Farber2004}
\bysame, \emph{{Instabilities of robot motion}}, Topology Appl. \textbf{140}
  (2004), 245--266.

\bibitem{Lubawski2014}
Wojciech Lubawski and Wacław Marzantowicz, \emph{{Invariant topological
  complexity}}, Bull. London Math. Soc. \textbf{47} (2014), 101--117.

\bibitem{Schwarz}
Albert~S. Schwarz, \emph{{The genus of a fiber space}}, Amer. Math. Soc.
  Transl. \textbf{55} (1966), 49--140.

\end{thebibliography}

\noindent\textsc{Zbigniew Błaszczyk}\\
\textit{Faculty of Mathematics and Computer Science\\
Adam Mickiewicz University\\
Umultowska 87\\
61-614 Poznań, Poland}\\
\texttt{blaszczyk@amu.edu.pl}\medskip

\noindent\textsc{Marek Kaluba}\\
\textit{Faculty of Mathematics and Computer Science\\
Adam Mickiewicz University\\
Umultowska 87\\
61-614 Poznań, Poland}\\
\texttt{kalmar@amu.edu.pl}
\end{document}